\documentclass[a4paper,10pt]{article}

\usepackage{newcent}

\usepackage[utf8]{inputenc}
\usepackage[english]{babel}
\usepackage{graphicx}
\usepackage[T1]{fontenc}
\usepackage{amsmath}
\usepackage{amssymb} 
\usepackage{amsthm}

\usepackage{ifthen,amsfonts,graphicx,latexsym,algorithm,algorithmic,url,txfonts}

\usepackage{mathabx}
\usepackage{bbm} 
\usepackage{tikz}
\usepackage[numbers]{natbib}
\usepackage{titlesec}
\usepackage[textwidth=480pt,textheight=650pt]{geometry}

\usepackage{setspace}

\usepackage[pdfborder={0 0 0}]{hyperref}

\numberwithin{equation}{section} 

\newcommand{\R}{\mathbb{R}}
\newcommand{\N}{\mathbb{N}}

\newcommand{\Prob}{\mathcal{P}}
\newcommand{\Leb}{\mathcal{L}} 
\newcommand{\A}{\mathcal{A}} 
\newcommand{\B}{\mathcal{B}} 
\newcommand{\K}{\mathcal{K}} 
\newcommand{\M}{\mathcal{M}}

\newcommand{\Id}{\mathrm{Id}}

\newcommand{\BV}{\mathrm{BV}}

\newcommand{\nO}{\mathbf{n}_\Omega} 

\newcommand{\ddr}{~\mathrm{d}} 
\newcommand{\dr}{\partial}

\newcommand{\1}{\mathbbm{1}} 

\newcommand{\dst}[1]{\displaystyle{#1}}

\newtheorem{theo}{Theorem}[section]
\newtheorem*{theo*}{Theorem}
\newtheorem{prop}[theo]{Proposition}
\newtheorem{crl}[theo]{Corollary}
\newtheorem{lm}[theo]{Lemma}
\newtheorem{defi}[theo]{Definition}
\newtheorem*{asmp*}{Assumptions}

\newcommand{\review}[1]{#1}

\title{New estimates on the regularity of the pressure\\ in density-constrained Mean Field Games}

\author{Hugo Lavenant\thanks{Laboratoire de Math\'ematiques d'Orsay, Univ. Paris-Sud, CNRS, Universit\'e Paris-Saclay, 91405 Orsay, cedex FRANCE, {\tt hugo.lavenant@math.u-psud.fr}}, \ Filippo Santambrogio\thanks{Institut Camille Jordan, Universit\'e Claude Bernard - Lyon 1, 43 boulevard du 11 novembre 1918, 
69622 Villeurbanne cedex
FRANCE, {\tt santambrogio@math.univ-lyon1.fr}}}

\date{\today}

\begin{document}

\maketitle

\begin{abstract} We consider variational Mean Field Games endowed with a constraint on the maximal density of the distribution of players. Minimizers of the variational formulation are equilibria for a game where both the running cost and the final cost of each player is augmented by a pressure effect, i.e. a positive cost concentrated on the set where the density saturates the constraint. Yet, this pressure is a priori only a measure and regularity is needed to give a precise meaning to its integral on trajectories. We improve, in the limited case where the Hamiltonian is quadratic, which allows to use optimal transport techniques after time-discretization, the results obtained in a paper of the second author with Cardaliaguet and M\'esz\'aros. We prove $H^1$ and $L^\infty$ regularity under very mild assumptions on the data, and explain the consequences for the MFG, in terms of the value function and of the Lagrangian equilibrium formulation.
\end{abstract}

\section{Introduction}

The main motivation of this paper is to provide improved regularity estimates about the pressure arising in a class of variational Mean Field Games (MFG) where the interaction between players is due to a density constraint $\rho\leq 1$ instead of arising from a penalization on the density itself. For the whole theory of Mean Field Games, a recent hot topic in applied mathematics introduced by Lasry and Lions in \cite{LL06cr2,LL07mf} and, independently, by Caines, Huang and Malam\'e in \cite{HMC}, we refer to the lecture notes by Cardaliaguet \cite{Carnotes} and to the video-recorded lectures by Lions, \cite{LLperso}. This theory is concerned with the behavior of a continuous family of rational agents, who need to choose a strategy on how to move in a domain where they meet other agents and their cost is affected by their presence. The goal is to study the Nash equilibria and characterize them in terms of PDEs. Most models assume stochastic effects on the trajectory of the agents, and the corresponding PDEs include diffusion terms which make the solution smooth and simplify the analysis, besides being reasonable from the modeling point of view. Analytically, the most difficult case consists in problems where the interaction between players is local (i.e. the cost at point $x$ and time $t$ depends on the value of the density $\rho_t(x)$, without averaging it in a neighborhood) and no diffusion is present. This case is essentially attacked when the game is of variational origin, i.e. it is a potential game, and the equilibrium condition arises as an optimality condition for an optimization problem in the class of density evolutions. For local potential MFG, we refer to \cite{Cardaliaguet2015,CardaliaguetG2015} and to the survey \cite{BenCarSan}.

A particular potential MFG has been studied in \cite{Cardaliaguet2016}, where the density $\rho$ is constrained to be below a certain threshold, which represents a capacity constraint of the transportation network or of the medium where agents move. In this case a pressure appears: according to what we know from fluid mechanics, the pressure is a scalar field, vanishing where the density does not saturate the constraint, and its gradient affects the accelaration of the particles. In terms of the equilibrium problem, the pressure is a price to pay to pass through saturated regions. This means that agents compute their total cost by integrating the pressure along their trajectory, but this generates some issues about its regularity, since it is a priori not well-defined on negligible sets such as curves.

Inspired by the considerations in \cite{af,af2}, in \cite{Cardaliaguet2016} two facts were proven. First, there is a way to define a precise representative $\hat p$ of the pressure $p$ such that almost every trajectory followed by the agents optimize a cost involving $\omega\mapsto \int_0^1 \hat p(t,\omega(t))dt$ among curves $\omega$ such that $\int_0^1 (Mp)(t,\omega(t))dt<+\infty$ ($Mp$ is the maximal function of $p$; such an integrability condition is required since $\hat p$ is defined as a limit of averages on balls and this is necessary to pass to the limit the averaged estimations). Second, the pressure $p$ belongs, under some assumptions on the data, to $L^2_{t,loc} BV_x$ (same regularity as the one obtained in \cite{af}), which guarantees $L^{1+\varepsilon}$ summability. This implies, by well-known harmonic analysis results, that $Mp$ is also summable to the same power, and guarantees that the class of curves satisfying the integrability condition on $Mp$ is large enough.

In the present paper this result is improved in several different ways, but we have to pay a price: we need to specialize to the case where the Hamiltonian is quadratic. This means that the cost payed by the agents following a trajectory $\omega$ is of the form
$$\int_0^1 \left(\frac12 |\dot{\omega}(t)|^2+\tilde V(t,\omega(t))\right)dt+\tilde\Psi(\omega(1)),$$
where $\tilde V=V+p$ is given by a running cost plus the pressure, and $\tilde \Psi$ is given by the final cost, augmented by a final pressure $P_1$. Note that we cannot exclude the presence of a pressure effect concentrated at the final time (the pressure should be considered as a measure with a singular part concentrated on $t=1$); we know that it is indeed possible to observe it in some particular cases, and this is indeed the role played by $P_1$ (see Sections \ref{heur} and \ref{pridual}).

In particular, the dependence in the velocity is quadratic, instead of using more general convex, and possibly space-dependent, functions $L(\omega(t),\dot{\omega}(t))$. This allows to use properties of optimal transport and of the Wasserstein distance $W_2$. In this precise setting, the results in \cite{Cardaliaguet2016} provided $L^2_{t,loc} BV_x$ regularity for the pressure if both the running cost $V$ and the final cost $\Psi$ were $C^{1,1}$; here (see Theorem \ref{theorem_main} for the precise statement) we manage to obtain $L^\infty_t H^1_x$ under the only assumption $V\in H^1$. Moreover, we obtain an inequality on the Laplacian of the pressure which, thanks to quite standard Moser iterations, provides $p\in L^\infty$ as soon as $V\in W^{1,q}$ for $q>d$ ($d$ being the space dimension). This boundedness is very important since it implies that the integrability condition on $Mp$ is always satisfied, and the class of competitor curves in the equilibrium condition includes now all curves. Finally, similar result are obtained for the singular pressure $P_1$ at time $t=1$, while \cite{Cardaliaguet2016} did not adress its behavior. As a last remark, the achievements of the present paper are global on a general bounded convex space domain, while the techniques in \cite{Cardaliaguet2016} could not easily be adapted to domains with boundary (they were presented in the torus; the technique can be adapted to other domains but only obtaining local results).

The main tool to obtain the desired estimate is an inequality on $\Delta (p+V)$, valid at any time $t$ on the saturated region where $p>0$, and all the estimates derive from this one (analogously, we also use a similar inequality on $\Delta(P_1+\Psi)$). We first provide (in Section \ref{heur}) an heuristic derivation of this inequality, based on the use of the convective derivative along the flow. In order to make the proof rigorous, the strategy is very much inspired from our previous work \cite{LS2017}, based on time-discretization, even if the inequalities we use are not the same.

As we said at the beginning, MFG with density constraints are the motivation for this work and they are the setting where these estimates show better their potential of applications. However, we believe that the technique deserves attention both for its remarkable simplicity, and for the possibility of being applied to other settings.

A similar setting can be found in the variational formulation by Brenier of the incompressible Euler equation. Yet, the sharp regularity of the pressure is in such a setting an open problem, as semi-concavity is a reasonable conjecture advanced by Brenier, but the current achievements do not go beyond the $L^2_t BV_x$ result mentioned above. Yet, due to the multiphasic nature of the problem formulated by Brenier, it is in general not possible to translate all the available techniques into such a more complicated setting (see for instance \cite{Lav}  where the time-convexity of the entropy is proven, but, differently from \cite{LS2017}, the same cannot be obtained for other internal energies; analogously, the same results of \cite{Lav} are also recovered in \cite{BarMon}, and the same algebraic obstruction prevents from generalizing the result to more general energies). On the other hand, the works on density-constrained MFG (including a first attempt, with a non-variational model, in \cite{modestproposal}) were inspired by previous works of the second author on crowd motion formulated as a gradient flow with density constraints (see \cite{Maury2010} and \cite{survey-crowd}), and the present technique seems possible to be applied to such a first-order (in time) setting. 
\review{Indeed, as explained in the core of the article, the technique of proof for the regularity of the final pressure $P_1$ is performed exactly as if we had a JKO scheme for a gradient flow (see \cite{JKO}).}

From the point of view of generalizations, in particular to other Hamiltonians, or to quadratic Hamiltonians on manifold (which could require to use curvature assumptions on the manifold), it seems that the main point is the computation of the Laplacian of the Hamilton-Jacobi equation: once suitable inequalities are available on it, the approach could be generalized. 

The paper is organized as follows. In the rest of Section 1 we present first the heuristic derivation of our estimates and then some consequences, in the framework of MFG, on the regularity of the value function. Section 2 presents useful preliminaries, then the context and the precise results we will prove, and finally the time-discrete approximation we choose. Section 3 contains the main estimates, divided into interior regularity for $p$ and then ``boundary'' regularity for $P_1$. Section 4 shows how to translate the discrete estimates into continuous ones, by providing limit results both on the primal and on the dual problem.

\subsection{Heuristic derivation of the estimates}\label{heur}

We start with the MFG system, which can be obtained as a consequence of the primal-dual optimality conditions of a variational problem, see \cite{Cardaliaguet2016}. These conditions read, for functions depending on time $t \in [0,1]$ and space $x \in \Omega$,
\begin{equation*}
\begin{cases}
\dr_t \rho - \nabla \cdot(\rho \nabla \phi)& = 0,  \\
- \dr_t \phi + \frac{1}{2} |\nabla \phi|^2 & \leq P+V \quad\mbox{(with equality on $\{\rho>0\}$)}, \\
\rho(0, \cdot) & = \bar{\rho}_0, \\
\phi(1, \cdot) & \leq \Psi\quad \mbox{(with equality on $\{\rho_1>0\}$)},
\end{cases}
\end{equation*} 
where $P\geq 0$ is a measure concentrated on the set $\{\rho=1\}$.
\review{As far as this heuristic justification is concerned, we will just look at the conditions which are satisfied on the support of $\rho$, where the inequalities become equalities. As we are interested in estimates on the pressure $P$, i.e. on the set $\{\rho=1\}$, this is a legit restriction.}
As we will see later, the pressure $P$ is a measure which can be decomposed into two parts: its restriction to $[0,1)\times \Omega$ is absolutely continuous w.r.t. the Lebesgue measure on $[0,1)\times \Omega$, and its density is denoted by $p$; on the other hand, there is also a part on $\{1\}\times \Omega$ which is singular, but absolutely continuous w.r.t. the Lebesgue measure on $\Omega$, and its density is denoted by $P_1$. This second part represents a jump of the function $\phi$ at $t=1$, which allows to re-write the system as follows.

\begin{equation}
\begin{cases}
\label{equation_MFG_system}
\dr_t \rho - \nabla \cdot(\rho \nabla \phi)& = 0,  \\
- \dr_t \phi + \frac{1}{2} |\nabla \phi|^2 & = p+V, \\
\rho(0, \cdot) & = \bar{\rho}_0, \\
\phi(1, \cdot) & = \Psi + P_1.
\end{cases}
\end{equation} 
where the density $\rho$ satisfies $\rho \leqslant 1$ everywhere and $p,P_1 \geqslant 0$ are strictly positive only on the regions where the constraint involving $\rho$ is saturated, i.e. where $\rho = 1$ ($\rho_1=1$ in the case of $P_1$). 

We denote by $D_t := \dr_t - \nabla \phi \cdot \nabla$ the convective derivative. The idea is to look at the quantity $- D_{tt} (\ln \rho)$. Indeed, the first equation of \eqref{equation_MFG_system} can be rewritten $D_t (\ln \rho) = \Delta \phi$. On the other hand, taking the Laplacian of the second equation in \eqref{equation_MFG_system}, it is easy to get, dropping a positive term, $-D_t (\Delta \phi) \leqslant \Delta (p+V)$. Hence, 
\begin{equation}
\label{equation_estimate_continuous}
- D_{tt}(\ln \rho) \leqslant \Delta (p+V).
\end{equation}  
Notice that if $\rho(\bar{t},\bar{x}) = 1$, then $\rho$ is maximal at $(\bar{t},\bar{x})$ hence $- D_{tt}(\ln \rho)(\bar{t},\bar{x}) \geqslant 0$. On the other hand, if $\rho(\bar{t},\bar{x}) < 1$ then $p(\bar{t},\bar{x}) = 0$. In other words, $p$ satisfies $\Delta p \geqslant - \Delta V$ on $\{ p > 0 \}$, which looks like an obstacle problem. Multiplying \eqref{equation_estimate_continuous} by $p$, integrating w.r.t. space at a given instant in time and doing an integration by parts, for all $t$ 
\begin{equation}
\label{equation_fundamental_heuristic}
\int_\Omega \nabla p(t, \cdot) \cdot \nabla (p(t, \cdot)+V) \leqslant \int_{\dr \Omega} p(t, \cdot) [\nabla (p(t, \cdot)+V) \cdot \nO], 
\end{equation}  
where $\nO$ is the outward normal to $\Omega$. 
\review{As $\nabla (p+V)$ is the acceleration of the agents, who are constrained to stay in $\Omega$, under the assumption that $\Omega$ is convex we get $\nabla (p(t, \cdot)+V) \cdot \nO \leqslant 0$, hence the l.h.s. of \eqref{equation_fundamental_heuristic} is negative.} 
From this we immediately see that $\| \nabla p(t, \cdot) \|_{L^2(\Omega)} \leqslant \| \nabla V \|_{L^2(\Omega)}$, i.e. that $p \in L^\infty((0,1) ; H^1(\Omega))$. Moreover, \review{taking $m > 1$ and mutliplying by $p^m$, and provided that $\nabla V \in L^q(\Omega)$ with $q > d$, we can use Moser iterations (i.e. iterating the inequality we obtain for different values of $m$) to prove that $p(t, \cdot) \in L^\infty(\Omega)$ with a norm depending only on $V$ and $\Omega$}. For the final pressure $P_1$, we only look at $D_t (\ln \rho) = \Delta \phi$. Using the equation for the terminal value of $\phi$, 
\begin{equation}
\label{equation_estimate_continuous_final}
D_{t}(\ln \rho)(1, \cdot) = \Delta (P_1+\Psi).
\end{equation}  
The l.h.s. is positive at every point $\bar{x}$ such that $\rho(1, \bar{x}) = 1$, hence we get $\Delta(\Psi + P_1) \geqslant 0$ on $\{ P_1 > 0 \}$. From exactly the same computations, we deduce $\| \nabla P_1 \|_{L^2(\Omega)} \leqslant \| \nabla \Psi \|_{L^2(\Omega)}$ and the $L^\infty(\Omega)$ norm of $P_1$ depends only on $\Omega$ and $\Psi$ provided that $\nabla \Psi \in L^q(\Omega)$ with $q > d$.  

Let us say that this strategy, namely looking at the convective derivative of quantities such as $\ln \rho$ was in fact already used by Loeper \cite{Loeper2006} to study a problem similar to ours (related to the reconstruction of the early universe), but in a case without potential and where $\Delta p := \rho - 1$. In his case, \eqref{equation_estimate_continuous} leads to a differential inequality involving only $\rho$ from which a $L^\infty$ bound on $\rho$ was deduced. 

\review{The assumption of convexity of $\Omega$ can be surprising, but is crucial for our method, as it was already the case in our previous work \cite{LS2017}. Roughly speaking, it prevents the interaction with the boundary from causing congestion: only the potentials $V$ and $\Psi$ are a source of congestion. As far as we can see, we do not know how to relax the assumption of convexity of $\Omega$ and still be able to control the r.h.s. of \eqref{equation_fundamental_heuristic}.}


\review{It is also possible to adapt these heuristic computations to the case of more general Hamiltonians. For instance, if we replace $\Omega$ by a Riemannian manifold (i.e. if we stick to a quadratic Lagrangian and Hamiltonian but insert a specific $x$-dependency), it is clear that the heuristic computation can be performed exactly in the same way provided that the manifold has a positive Ricci curvature, as the inequality involving the Laplacian of the Hamilton-Jacobi equation can be deduced from Bochner's formula. On the other hand, if we take a Lagrangian which is not quadratic, for instance of the form $L(\dot{\omega})$ (we omit explicit dependence on the point, not to overburden the computations), the mean field game system now reads (we have not included the temporal boundary conditions)
\begin{equation*}
\begin{cases}
\dr_t \rho - \nabla \cdot(\rho \nabla H(\nabla \phi))& = 0,  \\
- \dr_t \phi + H(\nabla \phi) & = p+V,
\end{cases}
\end{equation*} 
where $H$ is the Legendre transform of $L$.
Compared to the previous system \eqref{equation_MFG_system}, the velocity of the agents appearing in the continuity equation is given by $ - \nabla H(\nabla \phi)$ and the convective derivative is now $D_t = \dr_t - \nabla  H(\nabla\phi) \cdot \nabla$. Expanding the continuity equation leads to 
\begin{equation*}
D_t (\ln \rho) = \nabla \cdot ( \nabla  H(\nabla\phi) ).
\end{equation*}
For instance, for $L(\dot\omega)=|\dot\omega|^r/r$, we have $H(z)=|z|^q/q$, where $q$ is the conjugate exponent to $r$, and we have $D_t (\ln \rho) =\Delta_q(\phi)$. Because of the non-linearity of this term, as the reader can see, computations become much more involved than the quadratic case, but it is still possible to obtain an inequality of the form
$$\nabla\cdot (D^2H(\nabla\phi)\nabla(p+V))\geq 0,$$
which proves again that $p+V$ is a subsolution of a suitable elliptic equation. Unfortunately this equation, though being linear in $p+V$, can be highly degenerate if $H$ is not uniformly elliptic, and the analysis of the resulting estimates is far from the objectives of this paper.}

\subsection{On the consequences for the value function and the Lagrangian point of view}

Though they can be seen as interesting in themselves, the estimates obtained on the pressure have consequences for the interpretation of the MFG system. 

The first one has to do with the regularity of the value function $\phi$. Indeed, as understood in the works by Cardaliaguet and collaborators \cite{Cardaliaguet2015, CardaliaguetG2015, CardaliaguetPT2015}, a solution of a Hamilton-Jacobi equation exhibits regularity as soon as the r.h.s. is bounded from below and its positive part lies in $L^{1+d/2 + \varepsilon}$. In the aforementioned articles, such an assumption on the r.h.s. was obtained by assuming a moderate growth on the penalization of congestion in the primal problem. In the case studied in this article, where the density is forced to stay below a given threshold, the naive estimate on the r.h.s. of the Hamilton-Jacobi equation leads to a $L^1$ bound. The estimates obtained previously in \cite{Cardaliaguet2015} do not allow to deduce regularity of the value function (except if $d=1$). However, with what we prove in the present paper, one can deduce the following (see Section \ref{pridual} below for the definition of the dual and primal problem). 

\begin{prop}
Assume that either $V \in H^1(\Omega)$ and $d \leqslant 4$ or that $V \in W^{1,q}(\Omega)$ with $q > d$. Then there exists $(\bar{\phi}, \bar{P})$ a solution to the dual problem such that, for any $[0, T] \times \tilde{\Omega}$ compactly embedded in $[0,1) \times \mathring{\Omega}$, the function $\bar{\phi}$ is Hölder-continuous on $[0, T] \times \tilde{\Omega}$ and $\dr_t \bar{\phi} \in L^{1+\varepsilon}([0, T] \times \tilde{\Omega})$, $\nabla \bar{\phi} \in L^{2+\varepsilon}([0, T] \times \tilde{\Omega})$ for some $\varepsilon > 0$.  
\end{prop}    

\begin{proof}[Sketch of the proof]
The proof of this result can be obtained thanks to Theorem \ref{theorem_main} proved below. Indeed, if either one of \review{these} assumption is true, it implies (thanks to Sobolev injections) that the density $\bar p$ of $\bar P$ is such that $\bar{p}+V \in L^{r}([0, T] \times \tilde{\Omega})$ for some $r > 1+d/2$. Combined with \cite[Lemma 3.5]{Cardaliaguet2015} (for the Hölder regularity) and \cite[Theorem 1.1]{CardaliaguetPT2015} (for the Sobolev regularity), we deduce the result. 
\end{proof}

On the other hand, we can also deduce some information about the Lagrangian point of view, for instance when we have $L^\infty$ bounds on the pressure. The necessity for such bounds is explained in details in the introduction of our previous article \cite{LS2017}, let us just say that it has to do with how one chooses the correct representative for the pressure. We do not reproduce the derivation here, but just state the result. For that, we need to choose $\bar{P}$ a solution of the dual problem. We know from our theorem \ref{theorem_main} that it has a density that we will call $\bar p$ on $[0,1)\times\Omega$, and we select a precise representative of it, according to the formula in \cite{af2}, in the following way: 
\begin{equation}
\label{equation_p_limsup}
\hat{p}(t,x) = \limsup_{\varepsilon \to 0} \frac{1}{|B(x, \varepsilon)|} \int_{B(x,\varepsilon)} \bar{p}(t,y) \ddr y.
\end{equation}

Notice that if $V, \Psi \in W^{1,q}(\Omega)$, we know that $V,\Psi$ are continuous and (thanks to Theorem \ref{theorem_main}), that $\bar{p}$ is in $L^\infty((0,1)\times\Omega)$. Analogously, we also set
\begin{equation}
\label{equation_pp_limsup}
\hat{P}_1(x) = \limsup_{\varepsilon \to 0} \frac{1}{|B(x, \varepsilon)|} \int_{B(x,\varepsilon)} \bar{P}_1(y) \ddr y.
\end{equation}

 We will need to use the notation $e_t : \omega \in H^1([0,1], \Omega) \to \omega(t) \in \Omega$ which denotes the evaluation operator. 

\begin{prop}
\label{proposition_Nash_equilibrium}
\review{Assume that $V, \Psi \in W^{1,q}(\Omega)$, that $(\bar{\phi},\bar{P})$ is the solution of the dual problem given in Theorem \ref{theorem_main}, and that $\hat{p}, \hat{P}_1$ are defined everywhere by \eqref{equation_p_limsup} and \eqref{equation_pp_limsup}. }
Then there exists $Q \in \Prob(H^1([0,1], \Omega))$ a measure on the set of $H^1$ curves valued in $\Omega$ such that $\rho : t \to e_t \# Q$ is a solution of the primal problem and, for $Q$-a.e. $\gamma$, the curve $\gamma$ minimizes 
\begin{equation*}
\omega \to \frac{1}{2} \int_0^1 |\dot{\omega}(t)|^2 \ddr t + \int_0^1 \left( V(\omega(t)) + \hat{p}(t, \omega(t)) \right) \ddr t + \left( \Psi(\omega(1)) + \hat{P}_1(\omega(1)) \right)
\end{equation*}   
among all curves $\omega \in H^1([0,1], \Omega)$ such that $\omega(0) = \gamma(0)$.
\end{prop}

\noindent If we come back to the interpretation of the MFG system, $Q$ exactly describes the strategy of all the agents: it is the ditribution of mass on the possible strategies, and Proposition \ref{proposition_Nash_equilibrium} reads as the fact that $Q$ is a Nash-equilibrium, provided that the agents pay the prices $\bar{p}$ and $\bar P_1$ in the regions where the constraint $\rho \leqslant 1$ is saturated.  

\section{Notations, optimal transport and the variational problems}

In all the sequel, $\Omega$ will denote the closure of an open bounded convex domain of $\R^d$ with smooth boundary. We assume that the Lebesgue measure of $\Omega$, denoted by $|\Omega|$ is strictly larger than $1$: it will be necessary to get existence of probability measures on $\Omega$ with density bounded by $1$. The generalization to the case where $\Omega$ is the $d$-dimensional torus is straightforward and we do not address it explicitly. The space of probability measures on $\Omega$ will be denoted by $\Prob(\Omega)$. This space $\Prob(\Omega)$ is endowed with the weak-* topology, i.e. the topology coming from the duality with $C(\Omega)$, the continuous functions from $\Omega$ valued in $\R$.  

We will also make use of the space of positive measures on the product $[0,1] \times \Omega$ which we denote by $\M_+([0,1] \times \Omega)$.  

In all the sequel, when the element of integration is not specified, it is assumed tacitly that integration is performed w.r.t. the Lebesgue measure (either the $d$-dimensional Lebesgue measure on $\Omega$ or the $d+1$-dimensional Lebesgue measure on $[0,1] \times \Omega$). Similarly, a measure is said to be absolutely continuous if it is the case w.r.t. the Lebesgue measure.

\subsection{The Wasserstein space}

The space $\Prob(\Omega)$ of probability measures on $\Omega$ is endowed with the Wasserstein distance: if $\mu$ and $\nu$ are two elements of $\Prob(\Omega)$, the $2$-Wasserstein distance $W_2(\mu, \nu)$ between $\mu$ and $\nu$ is defined via
\begin{equation}
\label{equation_definition_Wasserstein_distance}
W_2(\mu, \nu) := \sqrt{ \min \left\{ \int_{\Omega \times \Omega} |x-y|^2 \ddr \gamma(x,y) \ :  \ \gamma \in \Prob(\Omega \times \Omega) \text{ and } \pi_0 \# \gamma = \mu, \ \pi_1 \# \gamma = \nu \right\}  }.
\end{equation} 
In the formula above, $\pi_0$ and $\pi_1 : \Omega \times \Omega \to \Omega$ stand for the projections on respectively the first and second component of $\Omega \times \Omega$. If $T : X \to Y$ is a measurable application and $\mu$ is a measure on $X$, then the image measure of $\mu$ by $T$, denoted by $T \# \mu$, is the measure defined on $Y$ by $(T \# \mu)(B) = \mu(T^{-1}(B))$ for any measurable set $B \subset Y$. For general results about optimal transport, the reader might refer to \cite{Villani2003} or \cite{OTAM}.

In all the sequel, we identify a measure with its density w.r.t. the Lebesgue measure on $\Omega$. Moreover, if $\mu \in \Prob(\Omega)$, we write $\mu \leqslant 1$ if the measure $\mu$ is absolutely continuous and its density is a.e. bounded by $1$. 

The Wasserstein distance admits a dual formulation, the dual variables being the so-called Kantorovich potentials. The main properties of these potentials, in the case which is of interest to us, are summarized in the proposition below. We restrict to the cases where the measures have a strictly positive \review{density} a.e., as in this particular case the potentials are unique (up to a global additive constant). The proof of these results can be found in (but not exclusively) \cite[Chapters 1 and 7]{OTAM}.  

\begin{prop}
\label{proposition_Kantorovicth_potential}
Let $\mu, \nu \in \Prob(\Omega)$ be two absolutely continuous probability measures with strictly positive density. Then there exists a unique (up to adding a constant to $\varphi$ and subtracting it from $\psi$) pair $(\varphi, \psi)$ of Kantorovich potentials satisfying the following properties. 
\begin{enumerate}
\item The squared Wasserstein distance $W_2^2(\mu,\nu)$ can be expressed as
\begin{equation*}
\frac{W_2^2(\mu, \nu)}{2}  = \int_\Omega \varphi \mu + \int_\Omega \psi \nu.
\end{equation*}
\item The "vertical" derivative of $W_2^2(\cdot, \nu)$ at $\mu$ is $\varphi$: if $\tilde{\rho} \in \Prob(\Omega)$ is any probability measure, then 
\begin{equation*}
\lim_{\varepsilon \to 0} \frac{W_2^2((1-\varepsilon)\mu + \varepsilon \tilde{\rho}, \nu) - W_2^2(\mu, \nu)}{2} = \int_\Omega \varphi (\tilde{\rho} - \mu). 
\end{equation*} 
\item The potentials $\varphi$ and $\psi$ are one the $c$-transform of the other, meaning that we have
\begin{equation*}
\begin{cases}
\varphi(x) & = \dst{\inf_{y \in \Omega} \left( \frac{|x-y|^2}{2} - \psi(y)  \right)} \\
\psi(y) & = \dst{\inf_{x \in \Omega} \left( \frac{|x-y|^2}{2} - \varphi(x)  \right)}.
\end{cases}
\end{equation*}
\item There holds $(\Id - \nabla \varphi) \# \mu = \nu$ and the transport plan $\gamma := (\Id, \Id - \nabla \varphi) \# \mu$ is optimal in the problem \eqref{equation_definition_Wasserstein_distance}.
\end{enumerate}
The function $\varphi$ (resp. $\psi$) is called the Kantorovitch potential from $\mu$ to $\nu$ (resp. from $\nu$ to $\mu$).  
\end{prop}

\subsection{Absolutely continuous curves in the Wasserstein space}

We will denote by $\Gamma$ the space of continuous curves from $[0,1]$ to $\Prob(\Omega)$.  This space will be equipped with the distance $d_\Gamma$ of the uniform convergence, i.e.
\begin{equation*}
d_\Gamma(\rho^1, \rho^2) := \max_{t \in [0,1]} W_2( \rho^1_t, \rho^2_t ).
\end{equation*} 

\noindent Following \cite[Definition 1.1.1]{Ambrosio2008}, we will introduce the following subset of $\Gamma$. 

\begin{defi}
We say that a curve $\rho \in \Gamma$ is $2$-absolutely continuous if there exists a function $a \in L^2([0,1])$ such that, for every $0 \leqslant t \leqslant s \leqslant 1 $,  
\begin{equation*}
W_2(\rho_t, \rho_s) \leqslant \int_t^s a(r) \ddr r.
\end{equation*}
\end{defi}

\noindent The main interest of this notion lies in the following theorem, which we recall. 

\begin{theo}
If $\rho \in \Gamma$ is a $2$-absolutely continuous curve, then the quantity 
\begin{equation*}
|\dot{\rho}_t| := \lim_{h \to 0} \frac{W_2(\rho_{t+h}, \rho_t)}{h}
\end{equation*}
exists and is finite for a.e. $t$. Moreover, 
\begin{equation}
\label{equation_representation_A_sup}
\int_0^1 |\dot{\rho}_t|^2 \ddr t = \sup_{N \geqslant 2} \ \ \sup_{0 \leqslant t_1 < t_2 < \ldots < t_N \leqslant 1} \ \ \sum_{k=2}^{N} \frac{W_2^2(\rho_{t_{k-1}}, \rho_{t_k})}{t_k - t_{k-1}}.
\end{equation}
\end{theo}

\begin{proof}
The first part is just \cite[Theorem 1.1.2]{Ambrosio2008}. The proof of the representation formula \eqref{equation_representation_A_sup} can easily be obtained by adapting the proof of \cite[Theorem 4.1.6]{Ambrosio2003}.  
\end{proof}

\noindent The quantity $|\dot{\rho}_t|$ is called the metric derivative of the curve $\rho$ and heuristically corresponds to the norm of the derivative of $\rho$ at time $t$ in the metric space $(\Prob(\Omega), W_2)$. The quantity $\int_0^1 |\dot{\rho}_t|^2 \ddr t$ behaves like a $H^1$ norm, see \cite[Proposition 2.7]{LS2017}, but we will not make a use of this.

\subsection{Primal and dual problem}\label{pridual}

To state our main theorem, we do the following assumptions, which will hold throughout the whole article. 

\begin{asmp*}
~
\begin{itemize}
\item[(A1)] The domain $\Omega$ is the closure of an open bounded convex subset of $\R^d$ with Lebesgue measure $|\Omega|>1$.
\item[(A2)] We fix $V \in H^1(\Omega)$ (the "running" cost) and assume that it is positive. 
\item[(A3)] We fix $\Psi \in H^1(\Omega)$ (the "final" cost ) and assume that it is positive. 
\item[(A4)] We take $\bar{\rho}_0 \in \Prob(\Omega)$ (the initial probability measure) such that $\bar{\rho}_0  \leqslant1$. 
\end{itemize}
\end{asmp*}

\noindent We denote by $\Gamma_0 \subset \Gamma$ the set of curves such $\rho \in \Gamma$ that $\rho_0 = \bar{\rho}_0$. As we will see below in the definition of the primal problem, it does not change anything to add a constant to $V$ or $\Psi$, hence (A2) and (A3) are equivalent to ask that $V$ and $\Psi$ are bounded from below. Note that the important assumption on $\Omega$ is its convexity, as already indicated by the formal computation in the introduction. 

The primal objective functional reads 

\begin{equation*}
\A(\rho) := \begin{cases}
\dst{  \int_0^1 \frac{1}{2} |\dot{\rho}_t|^2 \ddr t + \int_0^1 \left( \int_\Omega V  \rho_t  \right) \ddr t + \int_\Omega \Psi \rho_1 } & \text{if } \rho_t \leq 1 \text{ for all } t \in [0,1], \\
+ \infty & \text{else}.
\end{cases}
\end{equation*}

\begin{defi}
\label{definition_primal}
The primal problem is 
\begin{equation}
\label{equation_primal}
\min \left\{ \A(\rho) \ : \ \rho \in \Gamma_0 \right\}    
\end{equation}
\end{defi}

We will need to consider the dual of this problem. Let $\phi \in C^1([0,1] \times \Omega)$ and  $P \in C([0,1] \times \Omega)$ be smooth functions with $p$ positive and in such a way that the Hamilton Jacobi equation is satisfied as an inequality
\begin{equation}
\label{equation_HJ}
- \dr_t \phi + \frac{1}{2} |\nabla \phi|^2 \leqslant V + P
\end{equation}
and the final value of $\phi$ is constrained by
\begin{equation}
\label{equation_contrainte_finale}
\phi(1, \cdot) \leqslant \Psi.     
\end{equation}
The dual functional is defined as follows: 
\begin{equation*}
\B(\phi, P) := \int_\Omega \phi(0, \cdot)  \bar{\rho}_0 - \iint_{[0,1] \times \Omega} P.    
\end{equation*}

\noindent and there is no duality gap between the primal and the dual problem. However, to get existence of a solution of the dual problem, it is too restrictive to look only at smooth functions. As understood in \cite{Cardaliaguet2016}, the right functional space is the following. 

\begin{defi}
\label{definition_dual}
Let $\K$ be the set of pairs $(\phi, P)$ where $\phi \in \BV([0,1] \times \Omega) \cap L^2([0,1], H^1(\Omega))$ and $P \in \M_+([0,1] \times \Omega)$ is a positive measure, and the Hamilton Jacobi equation \eqref{equation_HJ} is understood in the distributional sense, provided we set $\phi(1^+, \cdot) = \Psi$ and that we take in account the possible jump from $\phi(1^-, \cdot)$ to $\phi(1^+, \cdot)$ in the temporal distributional derivative. \review{Specifically, we impose that for any smooth and non-negative function $f$ defined on $[0,1] \times \Omega$, 
\begin{equation*}
\iint_{[0,1] \times \Omega} \dr_t f \phi - \int_\Omega f(1, x)  \Psi(x) \ddr x + \int_\Omega f(0, x) \phi(0^+, x) \ddr x + \iint_{[0,1] \times \Omega} \frac{1}{2} f |\nabla \phi|^2  \leqslant \iint_{[0,1] \times \Omega}  f V  + \iint_{[0,1] \times \Omega} f \ddr P. 
\end{equation*}} 

For $(\phi, P) \in \K$, the dual functional is understood in the following sense:  
\begin{equation*}
\B(\phi, P) := \int_\Omega \phi(0^+, \cdot)  \bar{\rho}_0 - P([0,1] \times \Omega).    
\end{equation*}
\end{defi} 



\noindent Notice, in the definition of solutions of the Hamilton Jacobi equation we assume $\phi(0^-, \cdot) = \phi(0^+, \cdot)$ (no jump for $t=0$) but we set $\phi(1^+, \cdot) = \Psi$. The measure $P$ can have a part concentrated on $t=1$, which may lead to $\phi(1^-, \cdot) > \phi(1^+, \cdot) = \Psi$, provided the jump is compensated by the part of $P$ on $\{ 1 \} \times \Omega$. 

\begin{defi}
The dual problem is 
\begin{equation}
\label{equation_dual}
\max \left\{ \B(\phi, P) \ : \ (\phi, P) \in \K \right\}.    
\end{equation}
\end{defi}

\noindent These two problems are in duality in the following sense \cite[Propositions 3.3 and 3.8]{Cardaliaguet2016}. 

\begin{theo}
\label{theorem_duality}
There holds
\begin{equation*}
\min \left\{ \A(\rho) \ : \ \rho \in \Gamma_0 \right\}  = \max \left\{ \B(\phi, P) \ : \ (\phi, P) \in \K \right\}.  
\end{equation*}
\end{theo}

\noindent Notice that the existence of a solution to both the primal and the dual problem are included in this statement. The main result of this paper is the following: 

\begin{theo}
\label{theorem_main}
There exists a solution $(\bar{\phi}, \bar{P})$ of the dual problem such that: 
\begin{itemize}
\item[•] The restriction of $\bar{P}$ to $[0,1) \times \Omega$ has a density w.r.t. to the $d+1$-dimensional Lebesgue measure and this density, denoted by $\bar{p}$; satisfies $\| \nabla \bar{p}(t, \cdot) \|_{L^2(\Omega)} \leqslant \| \nabla V \|_{L^2(\Omega)}$ for a.e. $t \in [0,1]$. Moreover, if $V \in W^{1,q}(\Omega)$ with $q > d$, then $\| \bar{p} \|_{L^{\infty}([0,1) \times \Omega)} \leqslant C < + \infty$ with $C$ depending only on $\| \nabla V \|_{L^q(\Omega)}$ and $\Omega$. 
\item[•] The restriction of $\bar{P}$ to $\{ 1 \} \times \Omega$ has a density w.r.t. to the $d$-dimensional Lebesgue measure and this density (denoted by $\bar{P}_1$) satisfies $\| \nabla \bar{P}_1\|_{L^2(\Omega)} \leqslant \| \nabla \Psi \|_{L^2(\Omega)}$. Moreover, if $\Psi \in W^{1,q}(\Omega)$ with $q > d$, then $\| \bar{P}_1 \|_{L^{\infty}(\Omega)} \leqslant C < + \infty$ with $C$ depending only on $\| \nabla \Psi \|_{L^q(\Omega)}$ and $\Omega$. 
\end{itemize}
\end{theo}

\noindent As already understood in \cite[Section 5]{Cardaliaguet2016}, there are situations where the pressure is concentrated on $\{ 1 \} \times \Omega$: one cannot expect $\bar{P}$ to have a density with respect to the Lebesgue measure on the closed interval $[0,1]$. Nevertheless we prove in our theorem that the part of the pressure concentrated on $\{ 1 \} \times \Omega$ has spatial regularity, namely $H^1(\Omega)$ and even $L^\infty(\Omega)$ if $\Psi \in W^{1,q}(\Omega)$ with $q > d$. The rest of this paper is devoted to the proof of this theorem, the wrapping of the arguments being located at page \pageref{page_proof_theorem}. 

\subsection{The discrete problem}

To tackle this problem and make rigorous the estimate given in the introduction, we will approximate the primal problem in the following way: 
\begin{itemize}
\item We introduce a time-discretization. The integer $N+1$ denotes the number of time steps. The time step will be denoted by $\tau$ and we use the approximation 
\begin{equation*}
\int_0^1 \frac{1}{2} |\dot{\rho}_t|^2 \ddr t \simeq \sum_{k=0}^{N-1} \frac{W_2^2(\rho_{k \tau}, \rho_{(k+1) \tau})}{2 \tau}.     
\end{equation*}
    
\item We add an infinitesimal entropic penalization. The goal is to make sure that the density of the minimizers of the discrete problem will be bounded from below, which is necessary when we want to write the optimality conditions.

\item For technical reasons, we need to regularize $V$ and $\Psi$. We take $(V_N)_{N \in \N}$ a sequence which converges to $V$ in $H^1(\Omega)$ and such that $V_N$ is Lipschitz for any $N \geqslant 1$. We can assume moreover that $\| \nabla V_N \|_{L^2(\Omega)} \leqslant \| \nabla V \|_{L^2(\Omega)}$ and $V_N$ is positive. Similarly, we take a sequence $\Psi_N$ going to $\Psi$ in $H^1(\Omega)$ satisfying analogous properties. 
\end{itemize}

The entropic penalization will be realized with the help of the following functional: for any $\rho \in \Prob(\Omega)$, we set 
\begin{equation*}
H(\rho) := \begin{cases}
\dst{\int_\Omega \ln(\rho) \rho} & \text{if } \rho \text{ is absolutely continuous,} \\
+ \infty & \text{else}.
\end{cases}
\end{equation*}
It is known \cite[Chapter 7]{OTAM} that $H$ is lower semi-continuous on $\Prob(\Omega)$. Moreover, a simple application of Jensen's inequality yields
\begin{equation*}
- \ln(|\Omega|) \leqslant H(\rho) \leqslant 0
\end{equation*}
as soon as $\rho \leqslant 1$.

To define the discrete problem, we take $N \geq 1$ and denote by
\begin{equation*}
\Gamma_0^N := \{ (\rho_{k \tau})_{k \in \{0,1, \ldots, N \}} \ : \  \rho_{k \tau} \in \Prob(\Omega) \text{ and } \rho_0 = \bar{\rho}_0  \} \subset (\Prob(\Omega))^{N+1}   
\end{equation*}
the set of discrete curves starting from $\bar{\rho}_0$. We denote by $\tau := 1/N$ the time step. We choose $(\lambda_N)_{N \in \N}$ which goes to $0$ while being strictly positive, it will account for the scale of the entropic penalization. The speed at which $\lambda_N\to 0$ is irrelevant for the analysis, hence we do not need to specify it. The discrete functional $\A^{N}$ is defined on $\Gamma_0^N$ as 
\begin{align*}
&\A^{N}(\rho) := \\
& \begin{cases} 
\dst{\sum_{k=0}^{N-1} \frac{W_2^2(\rho_{k \tau}, \rho_{(k+1) \tau})}{2 \tau} + \sum_{k=1}^{N-1} \tau \left( \int_\Omega V_N \rho_{k \tau} + \lambda_N H(\rho_{k \tau})  \right) + \int_\Omega \Psi_N  \rho_{1} + \lambda_N H(\rho_1) } & \text{if } \rho_{k \tau} \leq 1 \text{ for all } k \in \{ 0,1, \ldots, N \}, \\ 
+ \infty & \text{else}.
\end{cases}
\end{align*}
The discrete problem reads as 
\begin{equation}
\label{equation_discrete_problem}
\min \left\{ \A^{N}(\rho) \ : \ \rho \in \Gamma_0^N \right\}.
\end{equation}

\begin{prop}
For any $N \geqslant 1$, there exists a unique solution to the discrete problem.
\end{prop}

\begin{proof}
The functional $\A^N$ is l.s.c. on $\Gamma^N_0$. Moreover, the curve $\rho$ which is constant and equal to $\bar{\rho}_0$ belongs to $\Gamma^N_0$ and is such that $\A^N(\rho) < + \infty$. As $\Gamma^N_0$ is compact (for the topology of the weak convergence of measures), the direct method of calculus of variations ensures the existence of a minimizer. 

Uniqueness clearly holds as $\lambda_N > 0$ and the entropy is a strictly convex function on $\Prob(\Omega)$.  
\end{proof}

\noindent From now on, for any $N \geqslant 1$, we fix $\bar{\rho}^{N}$ the unique solution of the discrete problem

\section{Estimates on the discrete problem}

Let us comment on a technical refinement: for some computations to be valid, we will need to assume that $\bar{\rho}_0$ is smooth is strictly positive. If it is not the case, it is easy to approximate (for fixed $N$) the measure $\bar{\rho}_0$ with a sequence $\bar{\rho}_0^{(n)}$ of smooth densities. For such a $\bar{\rho}_0^{(n)}$, the estimates obtained below for a given $N$ (Corollary \ref{corollary_bound_discrete}) do not depend on $n$. Hence it is easy to send $n$ to $+ \infty$, using the stability of the Kantorovich potentials \cite[Theorem 1.52]{OTAM} to see that these estimates are still satisfied for the solution of the discrete problem with initial condition $\bar{\rho}_0$. In short: we will do as if our initial condition $\bar{\rho}_0$ were smooth, and as long as the final estimates do not depend on the smoothness of $\bar{\rho}_0$ this will be legitimate.   

\subsection{Interior regularity}

We begin with the interior regularity. In this subsection, we fix $N \geq 1$ and $k \in \{ 1,2, \ldots, N-1 \}$ a given instant in time. We use the shortcut $\bar{\rho} := \bar{\rho}^{N}_{k \tau}$ and we also denote $\mu := \bar{\rho}^{N}_{(k-1) \tau}$ and $\nu := \bar{\rho}^{N}_{(k+1) \tau}$. As $\bar{\rho}^{N}$ is a solution of the discrete problem, we know that $\bar{\rho}$ is a minimizer, among all probability measures with density bounded by $1$, of 
\begin{equation*}
\rho \mapsto \frac{W_2^2(\mu, \rho) + W_2^2(\rho, \nu)}{2 \tau} + \tau \left( \int_\Omega V_N  \rho + \lambda_N H(\rho) \right).
\end{equation*} 

\begin{lm}
\label{lemma_positivity_rho}
The density $\bar{\rho}$ is strictly positive a.e.
\end{lm}

\begin{proof}
This is exactly the same proof as \cite[Lemma 3.1]{LS2017}, as the construction done in this proof preserves the constraint of having a density smaller than $1$. 
\end{proof}

\begin{prop}
\label{proposition_optimality_conditions}
Let us denote by $\varphi_\mu$ and $\varphi_\nu$ the Kantorovich potentials for the transport from $\bar{\rho}$ to $\mu$ and $\nu$ respectively. There exists $p \in L^1(\Omega)$, positive, such that $\{ p > 0 \} \subset \{ \bar{\rho} = 1 \}$ and a constant $C$ such that  
\begin{equation}
\label{equation_optimality_conditions}
\frac{\varphi_\mu + \varphi_\nu}{\tau^2} + V_N + p + \lambda_N \ln(\bar{\rho}) =  C \ \text{ a.e.}  
\end{equation}
Moreover $p$ and $\ln(\bar{\rho})$ are Lipschitz and $\nabla p \cdot \nabla \ln(\bar{\rho}) = 0$ a.e.
\end{prop}

\begin{proof}
Let $\tilde{\rho} \in \Prob(\Omega)$ such that $\tilde{\rho} \leq 1$. We define $\rho_\varepsilon := (1 - \varepsilon) \bar{\rho} + \varepsilon \tilde{\rho}$ and use it as a competitor. Clearly $\rho_\varepsilon \leqslant 1$, i.e. it is an admissible competitor. Comparing $\A^N(\rho_\varepsilon)$ to $\A^N(\rho)$, we extract the following information. Using Proposition \ref{proposition_Kantorovicth_potential}, as $\bar{\rho} > 0$, the Kantorovich potentials $\varphi_\mu$ and $\varphi_\nu$ are unique (up to a constant) and 
\begin{equation*}
\lim_{\varepsilon \to 0} \frac{W_2^2(\mu, \rho_\varepsilon) - W_2^2(\mu, \bar{\rho}) + W_2^2(\rho_\varepsilon, \nu) - W_2^2(\bar{\rho}, \nu)}{2 \tau^2} = \int_\Omega \frac{\varphi_\mu + \varphi_\nu}{\tau} (\tilde{\rho} - \bar{\rho}).
\end{equation*} 
The term involving $V_N$ is straightforward to handle as it is linear. The only remaining term is the one involving the entropy. But here, using the same reasoning as in \cite[Proposition 3.2]{LS2017}, we can say that 
\begin{equation*}
\limsup_{\varepsilon \to 0}  \frac{  H(\rho_\varepsilon) -  H(\bar{\rho})}{\varepsilon} \leqslant \int_\Omega \ln (\bar{\rho}) (\tilde{\rho} - \bar{\rho}).
\end{equation*} 
Putting the pieces together, we see that $\int_\Omega h \,(\tilde{\rho} - \bar{\rho}) \geqslant 0$ for any $\tilde{\rho} \in \Prob(\Omega)$ with $\tilde{\rho} \leqslant 1$, provided that $h$ is defined by 
\begin{equation*}
h := \frac{\varphi_\mu + \varphi_\nu}{\tau^2} + V_N + \lambda_N \ln (\bar{\rho}) 
\end{equation*}
It is known, analogously to \cite[Lemma 3.3]{Maury2010}, that this leads to the existence of a constant $C$ such that 
\begin{equation}
\label{equation_constant_pressure}
\begin{cases}
\bar{\rho} = 1 & \text{on } \{ h < C \} \\
\bar{\rho} \leqslant 1 & \text{on } \{ h = C \} \\
\bar{\rho} = 0 & \text{on } \{ h > C \}
\end{cases}  
\end{equation}
\review{Specifically, $C$ is defined as the smallest real $\tilde{C}$ such that $\Leb(\{ h \leqslant \tilde{C} \}) \geqslant 1$ ($\Leb$ being the Lebesgue measure on $\Omega$), and it is quite straightforward to check that this choice works.}
Note that the case $\{ h > C \}$ can be excluded by Lemma \ref{lemma_positivity_rho}. The pressure $p$ is defined as $p = (C - h)_+$, thus \eqref{equation_optimality_conditions} holds. It satisfies $p \geqslant 0$, and $\bar{\rho} < 1$ implies $p = 0$. 

It remains to answer the question of the integrability properties of $p$ and $\ln(\bar{\rho})$. Notice that $p$ is positive, and non zero only on $\{\bar{\rho} = 1 \}$. On the other hand, $\ln(\bar{\rho}) \leqslant 0$ and it is non zero only on $\{ \bar{\rho} < 1 \}$. Hence, one can write 
\begin{equation}
\label{equation_p_function_V}
p = \left( C - \frac{\varphi_\mu + \varphi_\nu}{\tau^2} + V_N \right)_+ \text{ and } \ln(\bar{\rho}) = - \frac{1}{\lambda_N} \left( C - \frac{\varphi_\mu + \varphi_\nu}{\tau^2} + V_N \right)_-.   
\end{equation}
Given that the Kantorovich potentials and $V_N$ are Lipschitz, it implies the Lipschitz regularity for $p$ and $\ln(\bar{\rho})$. Moreover, the identity $\nabla p \cdot \nabla \ln(\bar{\rho}) = 0$ is straightforward using $\nabla f_+ = \nabla f \1_{f > 0}$ a.e., which is valid for any $f \in H^1(\Omega)$. 
\end{proof}

\noindent Let us note that $\varphi_\mu$ and $\varphi_\nu$ have additional regularity properties, even though they depend heavily on $N$. 

\begin{lm}
\label{lemma_regularity_KP}
The Kantorovich potentials $\varphi_\mu$ and $\varphi_\nu$ belong to $C^{2,\alpha}(\mathring{\Omega}) \cap C^{1, \alpha}(\Omega)$.
\end{lm}

\begin{proof}
If $k \in \{ 2, \ldots, N \}$, thanks to Proposition \ref{proposition_optimality_conditions} (applied in $k-1$ and $k+1$), we know that $\mu$ and $\nu$ have a Lipschitz density and are bounded from below. Using the regularity theory for the Monge Ampère-equation \cite[Theorem 4.14]{Villani2003}, we can conclude that $\varphi_\mu$ and $\varphi_\nu$ belong to $C^{2,\alpha}(\mathring{\Omega}) \cap C^{1, \alpha}(\Omega)$.  
\end{proof}

\begin{theo}
\label{theorem_estimate_fundamental}
For any $m \geqslant 1$, the following inequality holds: 
\begin{equation}
\label{equation_estimate_fundamental}
\int_\Omega \nabla( p^m ) \cdot \nabla( p + V_N ) \leqslant 0.
\end{equation}
\end{theo}

\begin{proof}
The (optimal) transport map from $\bar{\rho}$ to $\mu$ is given by $\Id - \nabla \varphi_\mu$, and similarly for $\nu$. We consider the following quantity, (defined on the whole $\Omega$ given the regularity of $\mu, \nu, \varphi_\mu$ and $\varphi_\nu$), which is a discrete analogue of the l.h.s. of \eqref{equation_estimate_continuous}: 
\begin{equation*}
D(x) := - \frac{ \ln( \mu(x - \nabla \varphi_\mu(x)) ) + \ln( \nu(x - \nabla \varphi_\nu(x)) ) - 2 \ln (\bar{\rho}(x))   }{\tau^2}.    
\end{equation*}
Notice that if $\bar{\rho}(x) = 1$, then by the constraint $\mu(x - \nabla \varphi_\mu(x)) \leqslant 1$ and $\nu(x - \nabla \varphi_\nu(x)) \leqslant 1$ the quantity $D(x)$ is positive. On the other hand, using $(\Id - \nabla \varphi_\mu) \# \bar{\rho} = \mu$ and the Monge-Ampère equation, for all $x \in \mathring{\Omega}$ there holds
\begin{equation*}
\mu(x - \nabla_\mu \varphi_\mu(x)) = \frac{\bar{\rho}(x)}{\det(\Id - D^2 \varphi_\mu(x))},
\end{equation*}
and a similar identity holds for $\varphi_\nu$. Hence the quantity $D(x)$ is equal, for all $x \in \mathring{\Omega}$, to  
\begin{equation*}
D(x) =  \frac{ \ln( \det(\Id - D^2 \varphi_\mu(x))  ) + \ln( \det(\Id - D^2 \varphi_\nu(x))  )  }{\tau^2}.    
\end{equation*}
Diagonalizing the matrices $D^2 \varphi_\mu, D^2 \varphi_\nu$ and using the convexity inequality $ \ln(1 -y) \leqslant - y$, we end up with \begin{equation*}
D(x) \leqslant - \frac{\Delta ( \varphi_\mu(x) + \varphi_\nu(x) )}{\tau^2}.    
\end{equation*}
We multiply this identity by $p^m$ and integrate. Thanks to the fact that $D$ is positive on $\{ \bar{\rho} = 1 \}$, as $p$ is positive and does not vanish only on $\{ \bar{\rho} = 1 \}$, the quantity $p^m D$ is positive on $\mathring{\Omega}$. As the latter coincides, up to a Lebesgue negligible set, with $\Omega$, we get 
\begin{equation}
\label{equation_before_IPP}
\int_\Omega p^m \frac{\Delta ( \varphi_\mu + \varphi_\nu )}{\tau^2} \leqslant 0.      
\end{equation}
We do an integration by parts, which reads 
\begin{equation}
\label{equation_after_IPP}
\int_\Omega p^m \frac{\Delta ( \varphi_\mu + \varphi_\nu )}{\tau^2} = \int_{\dr \Omega} p^m \frac{\nabla \left( \varphi_\mu + \varphi_\nu \right)}{\tau^2} \cdot \nO - \int_\Omega  \nabla(p^m) \cdot \frac{\nabla \left( \varphi_\mu + \varphi_\nu \right)}{\tau^2}  
\end{equation}
To handle the boundary term, recall that $\nabla \varphi_\mu$ is continuous up to the boundary and that $x - \nabla \varphi_\mu(x) \in \Omega$ for every $x \in \Omega$ as $(\Id - \nabla \varphi_\mu) \# \bar{\rho} = \mu$. Given the convexity of $\Omega$, it implies $\nabla \varphi_\mu(x) \cdot \nO(x) \geqslant 0$ for every point $x \in \dr \Omega$ for which the outward normal $\nO(x)$ is defined. As it is the case for a.e. point of the boundary, as a similar inequality holds for $\varphi_\nu$, and given that $p^m$ is positive, we can drop the boundary term in \eqref{equation_after_IPP} and get
\begin{equation*}
\int_\Omega p^m \frac{\Delta ( \varphi_\mu + \varphi_\nu )}{\tau^2} \geqslant  - \int_\Omega  \nabla(p^m) \cdot \frac{\nabla \left( \varphi_\mu + \varphi_\nu \right)}{\tau^2}.  
\end{equation*}
\review{We emphasize that dropping this boundary term corresponds exactly to the heuristic inequality $\nabla (p(t, \cdot) + V) \cdot \nO \leqslant 0$ evoked in the introduction below equation \eqref{equation_fundamental_heuristic}.}
Using the optimality conditions \eqref{equation_optimality_conditions}, we see that 
\begin{equation*}
0 \geqslant \int_\Omega p^m \frac{\Delta ( \varphi_\mu + \varphi_\nu )}{\tau^2} \geqslant  \int_\Omega  \nabla(p^m) \cdot \nabla ( p + V_N + \lambda_N \ln (\bar{\rho}))  
\end{equation*}
Now remember that in Proposition \ref{proposition_optimality_conditions} we have proved that $\nabla p \cdot \nabla \ln(\bar{\rho}) = 0$ a.e., which is sufficient to drop the term involving $\nabla \ln(\bar{\rho})$ and get \eqref{equation_estimate_fundamental}.    
\end{proof} 

The inequation \eqref{equation_estimate_fundamental} implies the $H^1(\Omega)$ and $L^\infty(\Omega)$ regularity for the pressure: this can be seen as a consequence of Moser's regularity for elliptic equations \cite{Moser1960}. We still give the proof for the sake of completeness, and also because in the inequality \eqref{equation_estimate_fundamental}, the boundary terms have already been taken in account, which enables to get regularity up to the spatial boundary in a single set of iterations.

\begin{lm}
\label{lemma_Moser_iterations}
Let $f, W$ be Lipschitz functions defined on $\Omega$ such that $f$ vanishes on a set of measure at least $|\Omega|-1>0$ and such that, for any $m \geqslant 1$,
\begin{equation*}
\int_\Omega \nabla( f^m ) \cdot \nabla( f + W ) \leqslant 0.
\end{equation*} 
Then there holds $\| \nabla f \|_{L^2(\Omega)} \leqslant \| \nabla W \|_{L^2(\Omega)}$. Moreover, if $\nabla W \in L^q(\Omega)$ with $q > d$, then $f \in L^\infty(\Omega)$ and $\| f \|_{L^\infty(\Omega)}$ is bounded by a constant which depends only on $\Omega$ and $\| \nabla W \|_{L^q(\Omega)}$.
\end{lm}

\begin{proof}
With $m=1$ we immediately get
\begin{equation*}
\| \nabla f \|_{L^2(\Omega)} \leqslant \| \nabla W \|_{L^2(\Omega)}, 
\end{equation*} 
In particular, using the Poincaré inequality and the fact that $|\{ f= 0 \}| \geqslant |\Omega| - 1$, we see that $\| f \|_{L^1(\Omega)}$ is bounded by a constant depending only on $\Omega$ and $V$.

In the rest of the proof, we denote by $C$ a constant which depends only on $\Omega$ and $\| \nabla W \|_{L^q(\Omega)}$, and can change from line to line. We write the estimate, for any $m \geqslant 1$, as
\begin{equation*}
\int_\Omega |\nabla f|^2 f^{m-1} \leqslant - \int_\Omega (\nabla f \cdot \nabla W) f^{m-1}
\end{equation*}
Using Young's inequality, it is clear that
\begin{equation*}
\frac{2}{(m+1)^2} \int_\Omega \left| \nabla \left( f^{(m+1)/2} \right) \right|^2 = \frac{1}{2} \int_\Omega |\nabla f|^2 f^{m-1} \leqslant \frac{1}{2} \int_\Omega |\nabla W|^2 f^{m-1}.
\end{equation*}
Take $\tilde{\beta} < \beta < \frac{d}{d-2}$ sufficiently close to $\frac{d}{d-2}$ in such a way that $2 \tilde{\beta} / (\tilde{ \beta} -1) \leqslant q$. In particular, the $L^{2 \tilde{\beta} / (\tilde{ \beta} -1)}(\Omega)$ norm of $\nabla W$ is bounded by $ C \| \nabla W\|_{L^q(\Omega)}$. Moreover, we know that $H^1(\Omega) \hookrightarrow L^{2 \beta}(\Omega)$. Considering the fact that $f^{(m+1)/2}$ vanishes on a subset of measure at least $|\Omega| - 1$, it enables us to write \cite[Lemma 2]{Moser1960}  
\begin{align*}
\left( \int_\Omega f^{(m+1) \beta} \right)^{1/\beta} \leqslant C \int_\Omega \left| \nabla \left( f^{(m+1)/2} \right) \right|^2 & \leqslant C (m+1)^2 \int_\Omega |\nabla W|^2 f^{m-1} \\
& \leqslant C(m+1)^2 \left( \int_\Omega |\nabla W|^{ 2 \tilde{\beta} / (\tilde{ \beta} -1)} \right)^{(\tilde{ \beta} -1) / \tilde{\beta}} \left( \int_\Omega f^{(m-1) \tilde{\beta}} \right)^{1 / \tilde{\beta}},
\end{align*}
where the last inequality is Hölder's inequality with an exponent $\tilde{\beta}$. Thanks to this choice, taking the power $1/(m+1)$ on both sides, 
\begin{equation*}
\| f \|_{L^{(m+1) \beta}(\Omega)} \leqslant \left[ C (m+1)^2 \right]^{1/(m+1)} \| f \|_{L^{(m-1) \tilde{\beta}}}^{(m-1)/(m+1)}.
\end{equation*} 
It is easy to iterate this inequation. With $r = (m-1) \tilde{\beta}$, as $(m+1) \beta \geqslant \beta r / \tilde{\beta}$, one can write that
\begin{equation*}
\| f \|_{L^{\beta / \tilde{\beta} r}(\Omega)} \leqslant [C (r +1)]^{C/r} \max \left( \| f \|_{L^r(\Omega)}, 1 \right). 
\end{equation*} 
An easy induction (recall that we already know that $f$ is bounded in $L^1(\Omega)$ by a constant depending only on $\Omega$ and $W$) with $r_n = \left( \beta / \tilde{\beta} \right)^n$ shows that $\| f \|_{L^{ r_n}(\Omega)}$ is bounded by a constant which depends only on $\| \nabla W \|_{L^q(\Omega)}$ and $\Omega$, which implies the claimed $L^\infty(\Omega)$ bound. 
\end{proof}

\begin{crl}
\label{corollary_bound_discrete}
There holds $\| \nabla p \|_{L^2(\Omega)} \leqslant \| \nabla V \|_{L^2(\Omega)}$. Moreover, if $V \in W^{1,q}(\Omega)$ with $q > d$, then $p \in L^\infty(\Omega)$ and $\| p \|_{L^\infty(\Omega)}$ is bounded by a constant which depends only on $\Omega$ and $\| \nabla V \|_{L^q(\Omega)}$.
\end{crl}

\begin{proof}
It is enough to combine Lemma \ref{lemma_Moser_iterations} and \eqref{equation_estimate_fundamental}: one has to remember that $p$ vanishes where $\bar{\rho} = 1$, which is of measure at least $|\Omega| -1$, that $\| \nabla V_N \|_{L^2(\Omega)} \leqslant \| \nabla V \|_{L^2(\Omega)}$, and that $\| \nabla V_N \|_{L^q(\Omega)}$ is bounded independently on $N$ if $V \in W^{1,q}(\Omega)$.
\end{proof}

\subsection{Boundary regularity}

As we said, we will see that the pressure has a part which is concentrated on the time-boundary $t=1$.
The regularity of this part is proved exactly by the same technique than in the interior, hence we will only sketch it. In this subsection, we fix $N \geq 1$. We use the shortcut $\bar{\rho} := \bar{\rho}^{N}_{N \tau} = \bar{\rho}^{N}_{1}$ for the final measure and we also denote $\mu := \bar{\rho}^{N}_{(N-1) \tau}$. As $\bar{\rho}^{N}$ is a solution of the discrete problem, we know that $\bar{\rho}$ is a minimizer, among all probability measures with density bounded by $1$, of 
\begin{equation*}
\rho \mapsto \frac{W_2^2(\mu, \rho)}{2 \tau} +  \left( \int_\Omega \Psi_N \rho + \lambda_N H(\rho) \right).
\end{equation*} 
\review{This variational problem would be exactly the one obtained if we were to discretize the (Wasserstein) gradient flow of the functional $\Psi_N + \lambda_N H$ (with the constraint that the density does not exceed $1$) using the \emph{minimizing movement scheme} (known in this context as the JKO scheme \cite{JKO}). In particular, all the computations of this paragraph could be translated in the framework of gradient flows, i.e. first order evolutions in time.}

\begin{lm}
\label{lemma_positivity_rho_boundary}
The density $\bar{\rho}$ is strictly positive a.e.
\end{lm}

\begin{proof}
This property holds for exactly the same reason as in Lemma \ref{lemma_positivity_rho}.
\end{proof}

\begin{prop}
\label{proposition_optimality_conditions_boundary}
Let us denote by $\varphi_\mu$ the Kantorovich potential for the transport from $\bar{\rho}$ to $\mu$. There exists $p \in L^1(\Omega)$, positive, such that $\{ p > 0 \} \subset \{ \bar{\rho} = 1 \}$ and a constant $C$ such that  
\begin{equation}
\label{equation_optimality_conditions_boundary}
\frac{\varphi_\mu}{\tau} + \Psi_N + p + \lambda_N \ln(\bar{\rho}) =  C.  
\end{equation}
Moreover $p$ and $\ln(\bar{\rho})$ are Lipschitz and $\nabla p \cdot \nabla \ln(\bar{\rho}) = 0$ a.e.
\end{prop}

\begin{proof}
We use exactly the same competitor as in the proof of Proposition \ref{proposition_optimality_conditions}. It leads to the conclusion that $\int_\Omega h (\tilde{\rho} - \bar{\rho}) \geqslant 0$ for any $\tilde{\rho} \in \Prob(\Omega)$ with $\tilde{\rho} \leqslant 1$ where $h$ is defined as 
\begin{equation*}
h := \frac{\varphi_\mu}{\tau} + \Psi_N + \lambda_N \ln (\bar{\rho}).
\end{equation*}
It implies the existence of a constant $C$ such that \eqref{equation_constant_pressure} holds, and we define $p$ exactly in the same way, as $p := (C-h)_+$. The integrability properties of $p$ and $\ln (\bar{\rho})$ are derived in the same way as in the proof of Proposition \ref{proposition_optimality_conditions}. 
\end{proof}

\noindent The additional regularity for $\varphi_\mu$ is exactly the same than for the interior case (this is why we have also used an entropic penalization at the boundary).

\begin{lm}
The Kantorovich potential $\varphi_\mu$ belongs to $C^{2,\alpha}(\mathring{\Omega}) \cap C^{1, \alpha}(\Omega)$.
\end{lm}

\begin{theo}
\label{theorem_estimate_fundamental_boundary}
For any $m \geqslant 1$, the following inequality holds: 
\begin{equation}
\label{equation_estimate_fundamental_boundary}
\int_\Omega \nabla (p^m) \cdot \nabla (p+ \Psi_N) \leqslant 0. 
\end{equation}
\end{theo}

\begin{proof}
On the set $\mathring{\Omega}$ we consider the following quantity, which is the analogue of the l.h.s. of \eqref{equation_estimate_continuous_final}: 
\begin{equation*}
D(x) := \frac{\ln( \bar{\rho}(x)  ) - \ln ( \mu( x - \nabla \varphi_\mu(x)  )   )}{\tau}.
\end{equation*}
If $\bar{\rho}(x) = 1$, then by the constraint $\mu(x - \nabla \varphi_\mu(x)) \leqslant 1$ the quantity $D(x)$ is positive. On the other hand, exactly by the same estimate than in the proof of Theorem \ref{theorem_estimate_fundamental}, 
\begin{equation*}
D(x) \leqslant - \frac{(\Delta \varphi_\mu)(x)}{\tau}.
\end{equation*} 
We multiply this inequality by $p^m$, do an integration by parts (the boundary term is handled exactly as in Theorem \ref{theorem_estimate_fundamental}), and we end up with \eqref{equation_estimate_fundamental_boundary}.  
\end{proof}

\begin{crl}
\label{corollary_bound_discrete_boundary}
There holds $\| \nabla p \|_{L^2(\Omega)} \leqslant \| \nabla \Psi \|_{L^2(\Omega)}$. Moreover, if $\Psi \in W^{1,q}(\Omega)$ with $q > d$, then $p \in L^\infty(\Omega)$ and $\| p \|_{L^\infty(\Omega)}$ is bounded by a constant which depends only on $\Omega$ and $\| \nabla \Psi \|_{L^q(\Omega)}$.
\end{crl}

\begin{proof}
Exactly as in the proof of Corollary \ref{corollary_bound_discrete}, it is enough to combine Lemma \ref{lemma_Moser_iterations} and the estimate \eqref{equation_estimate_fundamental_boundary}.
\end{proof}

\section{Convergence to the continuous problem}

Recall that for any $N \geqslant 1$, $\bar{\rho}^{N}$ denotes the solution of the discrete problem.

\subsection{Convergence of the primal problem}

This convergence is very similar to the one performed in \cite{LS2017} hence we will not really reproduce it. Furthermore, as we are ultimately interested in the dual problem, we need only the convergence of the value of the primal problem, not of the minimizers.

Define $\tilde{\A}^N$ on $\Gamma_0^N$ exactly as the discrete primal functional $\A^N$, but where the regularized potentials $V_N$ and $\Psi_N$ are replaced by the true potentials $V$ and $\Psi$. Given the $L^\infty$ bound on $\rho$ (which holds if $\A^N$ or $\tilde{\A}^N$ are finite), one can see that for any $\rho \in \Gamma^N_0$ with density bounded by $1$,
\begin{equation}
\label{equation_difference_V_VN}
\left| \A^{N}(\rho) - \tilde{\A}^N(\rho) \right| \leqslant \| V - V_N \|_{L^1(\Omega)} + \| \Psi - \Psi_N \|_{L^1(\Omega)},
\end{equation}
and the r.h.s. goes to $0$ uniformly in $\rho$ as $N \to + \infty$.

On the other hand, using exactly the same proofs as in \cite{LS2017}, Section 5.1 and 5.2, one can easily check (the only thing to check is that all the constructions are compatible with the constraint of having a density bounded by $1$ but it is straightforward) that the value of the discrete problem 
\begin{equation*}
\min \left\{ \tilde{A}^N(\rho) \ : \ \rho \in \Gamma^N_0 \right\}
\end{equation*}
converges to the minimal value of the primal problem (notice that it is for this result that we need the scale $\lambda_N$ of the entropic penalization to go to $0$). Combined with \eqref{equation_difference_V_VN}, one can conclude the following. 
\begin{prop}
\label{proposition_convergence_primal}
The value of the discrete problem converges to the one of the continuous one in the sense that
\begin{equation*}
\lim_{N \to + \infty} \A^{N}(\bar{\rho}^{N})  = \min \left\{ \A(\rho) \ : \ \rho \in \Gamma_0 \right\}.
\end{equation*}
\end{prop}

\subsection{Convergence to the dual problem}

In this subsection, we want to build a value function $\phi^N$ which will go, as $N \to + \infty$, to a solution of the (continuous) dual problem. Notice that the discrete functional $\A^N$ is convex, hence we could consider discrete dual problem but we will not do it explicitly: indeed, the approximate value function $\phi^N$ will not be a solution of the discrete dual problem and we will not prove a duality result at the discrete level. 

On the contrary, we will just guess the expression of $\phi^N$ (we have to say to the inspiration for this kind of construction was found in the work of Loeper \cite{Loeper2006}) and use the explicit expression to prove that the value of some quantity which looks like the continuous dual objective, evaluated at $\phi^N$, is closed to the value of the discrete primal problem. Then, sending $N$ to $+ \infty$, we recover an admissible $(\bar{\phi}, \bar{P})$ for the continuous dual problem such that $\B(\bar{\phi}, \bar{P})$ is larger than the optimal value of the continuous primal problem (and this comes from estimates proved at the discrete level). It will allow us to conclude that $(\bar{\phi}, \bar{P})$ is a solution of the dual problem thanks to the absence of duality gap at the continuous level. Eventually, we pass to the limit the discrete estimates in Corollary \ref{corollary_bound_discrete} and Corollary \ref{corollary_bound_discrete_boundary} to get the ones for $\bar{p}$ and $\bar{P}_1$.

%

Let us recall that $\bar{\rho}^{N}$ is the solution of the discrete problem. For any $k \in \{ 0,1, \ldots, N-1 \}$, we choose $(\varphi^{N}_{k \tau}, \psi^{N}_{k \tau})$ a pair of Kantorovich potential between $\bar{\rho}^{N}_{k \tau}$ and $\bar{\rho}^{N}_{(k+1) \tau}$, such choice is unique up to an additive constant. According to Proposition \ref{proposition_optimality_conditions} and Proposition \ref{proposition_optimality_conditions_boundary}, making the dependence on $N$ and $k$ explicit, for any $k \in \{ 1,2, \ldots, N \}$, there exists a pressure $p^{N}_{k \tau}$ positive and Lipschitz, and a constant $C^{N}_{k \tau}$ such that
\begin{equation}
\label{equation_optimality_conditions_indexed}
\begin{cases}
\dst{ \frac{\psi^{N}_{(k-1) \tau} + \varphi^{N}_{k \tau}}{\tau^2} + V_N + p^{N}_{k \tau} + \lambda_N \ln(\bar{\rho}^{N}_{k \tau}) =  C^{N}_{k \tau}  } & k \in \{ 1,2, \ldots, N-1 \}, \\
\dst{ \frac{\psi^{N}_{(k-1) \tau} }{\tau} + \Psi_N + p^{N}_1 + \lambda_N \ln(\bar{\rho}^{N}_{k \tau}) =  C^{N}_{1}  } & k = N. 
\end{cases}
\end{equation} 

We define the following value function, defined on the whole interval $[0,1]$ which can be thought as a function which looks like a solution of what could be called a discrete dual problem.

\begin{defi}
\label{definition_value_function_approximate}
Let $\phi^{N} \in \BV([0,1] \times \Omega) \cap L^2([0,1] \times H^1(\Omega))$ the function defined as follows. The "final" value is given by 
\begin{equation}
\label{equation_value_function_boundary}
\phi^{N}(1^-,\cdot) := \Psi_N +  p^{N}_{1}.
\end{equation}
Provided that the value $\phi^{N}((k \tau)^-, \cdot)$ is defined for some $k \in \{ 1,2, \ldots, N \}$, the value of $\phi^{N}$ on $((k-1) \tau, k \tau ) \times \Omega$ is defined by 
\begin{equation}
\label{equation_value_function_inside}
\phi^{N}(t,x) := \inf_{y \in \Omega} \left( \frac{|x-y|^2}{2 (k \tau - t)} + \phi^{N}((k \tau)^-, y) \right).
\end{equation}
If $k \in \{ 1,2, \ldots, N-1 \}$, the function $\phi^{N}$ has a temporal jump at $t = k \tau$ defined by 
\begin{equation}
\label{equation_value_function_jump}
\phi^{N}((k \tau)^-, x) := \phi^{N}((k \tau)^+, x) + \tau \left( V_N + p^{N}_{k \tau} \right)(x) 
\end{equation}
\end{defi}

\noindent Notice that we have not included the entropic term: its only effect would have been to decrease $\phi^{N}$ (which in the end decreases the value of the dual functional) and it would have prevented us from getting compactness on the sequence $\phi^{N}$. The link between this value function and the Kantorovich potentials is the following.

\begin{lm}
\label{lemma_link_value_potential}
For any $k \in \{ 1,2, \ldots, N \}$, one has 
\begin{equation}
\label{equation_link_value_potential_-}
\phi^{N}((k \tau)^-, \cdot) \geqslant C^{N}_1 + \tau \sum_{j=k}^{N-1} C^{N}_{j \tau}  - \frac{\psi^{N}_{(k-1) \tau}}{\tau}.
\end{equation}
For any $k \in \{ 0,1, \ldots, N-1\}$, one has 
\begin{equation}
\label{equation_link_value_potential_+}
\phi^{N}((k \tau)^+, \cdot) \geqslant C^{N}_1 + \tau \sum_{j=k+1}^{N-1} C^{N}_{j \tau}  + \frac{\varphi^{N}_{k \tau}}{\tau}.
\end{equation}
\end{lm}

\begin{proof}
We will prove it by (decreasing) induction on $k \in \{ 0,1, \ldots, N \}$. For $k = N$, by the optimality conditions \eqref{equation_optimality_conditions_indexed} and the fact that $\ln(\bar{\rho}^{N}_1) \leqslant 0$, it is clear that \eqref{equation_link_value_potential_-} holds. 

Now assume that \eqref{equation_link_value_potential_-} holds for some $k$. Using \eqref{equation_value_function_inside}, one has 
\begin{align*}
\phi^{N}(((k-1) \tau)^+,x) & = \inf_{y \in \Omega} \left( \frac{|x-y|^2}{2 \tau} + \phi^{N}((k \tau)^-, y) \right) \\
& \geqslant C^{N}_1 + \tau \sum_{j=k}^{N-1} C^{N}_{j \tau} + \inf_{y \in \Omega} \left( \frac{|x-y|^2}{2 \tau} - \frac{\psi^{N}_{(k-1) \tau}(y)}{\tau}  \right) \\
& = C^{N}_1 + \tau \sum_{j=k}^{N-1} C^{N}_{j \tau} + \frac{1}{\tau} \inf_{y \in \Omega} \left( \frac{|x-y|^2}{2} - \psi^{N}_{(k-1) \tau}(y)  \right)
=  C^{N}_1 + \tau \sum_{j=k}^{N-1} C^{N}_{j \tau} + \frac{\varphi^{N}_{(k-1)\tau}(x)}{\tau},  
\end{align*}   
where the last inequality comes from the fact that $\varphi^{N}_{(k-1)\tau}$ is the $c$-transform of $\psi^{N}_{(k-1) \tau}$. This gives us \eqref{equation_link_value_potential_+} for $k-1$. On the other hand, assume that \eqref{equation_link_value_potential_+} holds for some $k$. Using \eqref{equation_value_function_jump} and the optimality conditions \eqref{equation_optimality_conditions_indexed} , 
\begin{align*}
\phi^{N}((k \tau)^-, x) & = \phi^{N}((k \tau)^+, x) + \tau \left( V_N + p^{N}_{k \tau} \right) \\ 
& \geqslant C^{N}_1 + \tau \sum_{j=k+1}^{N-1} C^{N}_{j \tau}  + \frac{\varphi^{N}_{k \tau}}{\tau} + \tau \left( V_N + p^{N}_{k \tau}  \right) \\
& = C^{N}_1 + \tau \sum_{j=k+1}^{N-1} C^{N}_{j \tau}  + C^{N}_{k \tau} - \frac{\psi^{N}_{k \tau}}{\tau} - \lambda_N \tau \ln(\bar{\rho}^{N}_{k \tau}) \geqslant C^{N}_1 + \tau \sum_{j=k}^{N-1} C^{N}_{j \tau}  - \frac{\psi^{N}_{k \tau}}{\tau}, 
\end{align*}  
which means that \eqref{equation_link_value_potential_-} holds for $k$. 
\end{proof}

\noindent From this identity, we can express some kind of duality result at the discrete level, which reads as follows. 

\begin{prop}
\label{proposition_duality_discrete}
For $N \geqslant 1$, the following inequality holds: 
\begin{equation}
\label{equation_duality_discrete}
\A^{N}(\bar{\rho}^{N}) \leqslant \int_\Omega \phi^{N}(0^+, \cdot) \bar{\rho}_0 - \tau \sum_{k=1}^{N-1} \int_\Omega p^{N}_{k \tau} - \int_\Omega p^{N}_1.
\end{equation}
\end{prop}

\noindent We have an inequality and not an equality because we have not included the entropic terms in the value function. 

\begin{proof}
The idea is to evaluate $\A^{N}(\bar{\rho}^{N})$ by expressing the Wasserstein distances with the help of the Kantorovich potentials.
\begin{align*}
\A^{N}(\bar{\rho}^{N}) & = \sum_{k=0}^{N-1} \frac{W_2^2(\bar{\rho}^{N}_{k \tau}, \bar{\rho}^{N}_{(k+1) \tau})}{\tau} + \sum_{k=1}^{N-1} \tau \left( \int_\Omega V_N \bar{\rho}^{N}_{k \tau} + \lambda_N H(\bar{\rho}^{N}_{k \tau})  \right) + \int_\Omega \Psi_N  \bar{\rho}^{N}_{1} + \lambda_N H(\bar{\rho}^{N}_1) \\
& = \frac{1}{\tau} \sum_{k=0}^{N-1} \left( \int_\Omega \varphi^{N}_{k \tau} \bar{\rho}^{N}_{k \tau} + \int_\Omega \psi^{N}_{k \tau} \bar{\rho}^{N}_{(k+1) \tau} \right)   + \sum_{k=1}^{N-1} \tau \left( \int_\Omega V_N  \bar{\rho}^{N}_{k \tau} + \lambda_N H(\bar{\rho}^{N}_{k \tau})  \right) + \int_\Omega \Psi_N  \bar{\rho}^{N}_{1} + \lambda_N H(\bar{\rho}^{N}_1) \\
& = \frac{1}{\tau} \int_\Omega \varphi^{N}_0 \bar{\rho}_0 + \sum_{k=1}^{N-1} \int_\Omega \left( \frac{\varphi^{N}_{k \tau} + \psi^{N}_{(k-1) \tau}}{2 \tau} + \tau(V_N + \lambda_N \ln (\bar{\rho}^{N}_{k \tau})) \right) \bar{\rho}^{N}_{k \tau} + \int_\Omega \left( \frac{\psi^{N}_{(N-1) \tau}}{2 \tau} + \Psi_N + \lambda_N \ln (\bar{\rho}^{N}_{1}) \right)  \bar{\rho}^{N}_1,
\end{align*}
where the last equality comes from a reindexing of the sums. Now we use the optimality conditions \eqref{equation_optimality_conditions_indexed} to handle the second and third term. Notice that, as $p^{N}_{k \tau}$ lives only where $\bar{\rho}^{N}_{k \tau} = 1$, that we can replace $\bar{\rho}^{N}_{k \tau}$ by $1$ when it is multiplied by the pressure. Recall also that the probability distributions, when integrated against a constant, are equal to this constant. We are left with 
\begin{align*}
\A^{N}(\bar{\rho}^{N}) & = \frac{1}{\tau} \int_\Omega \varphi^{N}_0 \bar{\rho}_0 + \sum_{k=1}^{N-1} \left( C^{N}_{k \tau} - \tau \int_\Omega p^{N}_{k \tau} \right)  +  \left( C^{N}_1 - \int_\Omega p^{N}_1 \right) \\
& \leqslant \int_\Omega \phi^{N}(0^+, \cdot) \bar{\rho}_0 - \tau \sum_{k=1}^{N-1} \int_\Omega p^{N}_{k \tau} - \int_\Omega p^{N}_1,  
\end{align*}
where the last equality comes from Lemma \ref{lemma_link_value_potential} which allows to make the link between the Kantorovich potential $\varphi^{N}_0$ and $\phi^{N}(0^+, \cdot)$.
\end{proof}

We want to pass to the limit $N \to + \infty$. To this extent, we rely on the fact that $\phi^{N}$ satisfies an explicit equation in the sense of distributions. We start to define the distribution which will be the r.h.s. of the Hamilton Jacobi equation. 

\begin{defi}
Let $\alpha^{N}$ and $P^{N}$ the positive measures on $[0,1] \times \Omega$ defined as 
\begin{equation*}
\begin{cases}
\alpha^{N} & := \dst{\tau \sum_{k=1}^{N-1} \delta_{t = k \tau} (p^{N}_{k \tau} + V_N) + \delta_{t=1} p^{N}_1}, \\
P^{N} & := \dst{\tau \sum_{k=1}^{N-1} \delta_{t = k \tau} p^{N}_{k \tau}  + \delta_{t=1} p^{N}_1}. 
\end{cases}
\end{equation*}
More precisely, for any test function $a \in C([0, 1] \times \Omega)$, 
\begin{equation*}
\int_{[0,1] \times \Omega} a \ddr \alpha^{N} := \tau \sum_{k=1}^{N-1} \int_\Omega a(k \tau, \cdot) \left( V_N + p^{N}_{k \tau}  \right) + \int_\Omega a(1, \cdot)  p^{N}_{1},
\end{equation*}
and similarly for $P^{N}$.
\end{defi}

\noindent In other words, $\alpha^{N}$ is, from the temporal point of view, a sum of delta function, each of them corresponding to the jump of the value function $\phi^{N}$. 

\begin{prop}
Provided that we set $\phi^{N}(0^-, \cdot) = \phi^{N}(0^+, \cdot)$ and $\phi^{N}(1^+, \cdot) = \Psi_N$, the following equation holds in the sense of distributions on $[0,1] \times \Omega$:
\begin{equation}
\label{equation_HJ_discrete}
- \dr_t \phi^{N} + \frac{1}{2} |\nabla \phi^{N}|^2 \leqslant \alpha^{N}.
\end{equation}
\end{prop}

\begin{proof}
As the pressures and the potentials $V_N, \Psi_N$ are Lipschitz, for any $t \in [0,1]$, the value function $\phi^N(t^+, \cdot)$ and $\phi^N(t^-, \cdot)$ are Lipschitz (but with a Lipschitz constant which may diverge as $N \to + \infty$). 

Notice that on each interval $(((k-1) \tau)^+, (k \tau)^-)$, the function $\phi^{N}$ is defined by the Hopf-Lax formula, hence solves the Hamilton-Jacobi equation $- \dr_t \phi^{N} + \frac{1}{2} |\nabla \phi^{N}|^2 = 0$ a.e. \cite[Section 3.3]{Evans2010}. It implies that the inequality $- \dr_t \phi^{N} + \frac{1}{2} |\nabla \phi^{N}|^2 \leqslant 0$ is also satisfied in the sense of distributions, as $\nabla \phi^N$ is bounded and $\dr_t \phi$ may have some singular parts, but they are positive.  

Provided that we set $\phi^{N}(0^-, \cdot) = \phi^{N}(0^+, \cdot)$ and $\phi^{N}(1^+, \cdot) = \Psi_N$, the measure $\dr_t \phi^N$ has a singular negative part at $\{ \tau, 2 \tau, \ldots, 1 \}$ corresponding to the jumps of the function $\phi^N$; but, given \eqref{equation_value_function_boundary} and \eqref{equation_value_function_jump}, the negative part of $\dr_t \phi^N$ is exactly $- \alpha^N$. 
\end{proof}

%

\noindent The next step is to pass to the limit $N \to + \infty$. To this extent, we need uniform bounds on $\alpha^N$, which derive easily from the bounds that we have on the pressure. 

\begin{lm}
\label{lemma_bound_p_approximate}
There exists a constant $C$, independent of $N$, such that $\alpha^{N}([0,1] \times \Omega) \leqslant C$ and $P^{N}([0,1] \times \Omega) \leqslant C$. 
\end{lm}

\noindent Recall that both $\alpha^N$ and $P^{N}$ are positive measures as we have chosen $V_N$ in such a way that it is positive.

\begin{proof}
We know that the $p^N_{k\tau}$, for $k \in \{ 1,2, \ldots, N \}$ have a gradient which is bounded uniformly in $L^2(\Omega)$. As moreover they all vanish on a set of measure at least $|\Omega| - 1$, they are bounded uniformly (w.r.t. $N$) in $L^1(\Omega)$. This is enough, in order to get the uniform bound on $P^{N}$. Given the way $V_N$ is built, the one for $\alpha^N$ is a straightforward consequence of the one on $P^N$.   
\end{proof}

\noindent Now that we have a bound on $\alpha^N$, to get compactness on the sequence $\phi^{N}$, we use the same kind of estimates used to prove existence of a solution in the dual at the continuous level, see for instance \cite[Section 3]{Cardaliaguet2016}. \review{We recall that $\K$, the set of admissible competitors for the dual problem, was defined in Definition \ref{definition_dual}.}  

\begin{prop}
There exists $(\bar{\phi}, \bar{P}) \in \K$ admissible for the dual problem such that 
\begin{equation*}
\begin{cases}
\dst{\lim_{N \to + \infty} \phi^{N}} = \bar{\phi} & \text{weakly in } \BV([0,1] \times \Omega) \cap L^2([0,1], H^1(\Omega)), \\
\dst{\lim_{N \to + \infty} P^{N}} = \bar{P} & \text{in } \M_+([0,1] \times \Omega). 
\end{cases}
\end{equation*}
\end{prop}

\begin{proof}
Given Lemma \ref{lemma_bound_p_approximate}, we know that $P^{N}$ is bounded in $\M_+([0,1] \times \Omega)$ independently of $N$. \review{Up to the extraction of a subsequence}, it converges weakly as a measure to some $\bar{P}$. On the other hand, once we know this convergence, it is easy to see that $\alpha^{N}$ converges as a measure on $\M_+([0,1] \times \Omega)$ to $\bar{P} + V$. 

We have assumed that $V$ and $\Psi$ are positive, and so are $V_N$ and $\Psi_N$, independently of $N$. Using the definition of $\phi^{N}$ and the positivity of the pressures, it is not hard to see that $\phi^{N}$ is positive $[0,1] \times \Omega$. Integrating the Hamilton Jacobi equation w.r.t. space and time and using the bound on $\alpha^{N}$ (Lemma \ref{lemma_bound_p_approximate}), we see that 
\begin{equation}
\label{equation_HJ_integrated}
\int_\Omega \phi^{N}(0^+, \cdot) - \int_\Omega \Psi_N + \frac{1}{2} \iint_{[0,1] \times
 \Omega} |\nabla \phi^{N}|^2 \leqslant C. 
\end{equation} 
Combined with the positivity of $\phi^{N}(0^+, \cdot)$ and a $L^1(\Omega)$ bound on $\Psi_N$, we see that $\nabla \phi^{N}$ is uniformly bounded in $L^2([0,1] \times \Omega)$. 

It remains to get a bound on $\dr_t \phi^{N}$. Of course, as a measure, it can be decomposed as a positive and a negative part. The negative part is concentrated on the instants $\{ \tau, 2 \tau, \ldots, 1 \}$ as $\dr_t \phi^N \geqslant 0$ on the intervals $(((k-1) \tau)^+, (k \tau)^-)$. On the other hand, on $\{ \tau, 2 \tau, \ldots, 1 \}$, the temporal derivative $\dr_t \phi^N$ coincides with $- \alpha^N$, hence the negative part is bounded as a measure. On the other hand, given that
\begin{equation*}
\iint_{[0,1] \times \Omega} \dr_t \phi^N = \int_\Omega \Psi_N - \int_\Omega \phi^N(0^+, \cdot) \leqslant \int_\Omega \Psi_N
\end{equation*}
is bounded independently of $N$, we see that $(\dr_t \phi^N)_+ = \dr_t \phi^N + (\dr_t \phi^N)_-$ is also bounded as a measure.

As a consequence, \review{up to the extraction of a subsequence} we know that $\phi^{N}$ converges weakly in $\BV([0,1] \times \Omega) \cap L^2([0,1], H^1(\Omega))$ to some $\bar{\phi}$. This convergence allows easily to pass to the limit in the Hamilton-Jacobi equation satisfied (in the sense of distributions) by $\phi^{N}$, hence $(\bar{\phi}, \bar{P})$ is admissible in the dual problem.   
\end{proof}

\noindent The last step, to show the optimality of the limit $(\bar{\phi}, \bar{P})$, is to pass to the limit in \eqref{equation_duality_discrete}.  

\begin{prop}
The pair $(\bar{\phi}, \bar{P}) \in \K$ is a solution of the dual problem.  
\end{prop} 

\begin{proof}
We have already proved in Proposition \ref{proposition_convergence_primal} that
\begin{equation*}
\lim_{N \to + \infty} \A^{N}(\bar{\rho}^{N})  = \min \left\{ \A(\rho) \ : \ \rho \in \Gamma_0 \right\}.   
\end{equation*} 
Given \eqref{equation_duality_discrete} and the duality result which holds for the continuous problem (Theorem \ref{theorem_duality}), it is enough to show that 
\begin{equation*}
\limsup_{N \to + \infty} \left( \int_\Omega \phi^{N}(0^+, \cdot) \bar{\rho}_0 - \tau \sum_{k=1}^{N-1} \int_\Omega p^{N}_{k \tau} - \int_\Omega p^{N}_1 \right) \leqslant \B(\bar{\phi}, \bar{P}) = \int_\Omega \bar{\phi}(0^+, \cdot) \bar{\rho}_0 - \bar{p}([0,1] \times \Omega).
\end{equation*} 
The convergence of the term involving the pressure is quite easy to show. Indeed, given the positivity of the pressures, 
\begin{equation*}
\tau \sum_{k=1}^{N-1} \int_\Omega p^{N}_{k \tau} + \int_\Omega p^{N}_1 = P^{N} ([0,1] \times \Omega) \to \bar{p}([0,1] \times \Omega)
\end{equation*}
by weak convergence. On the other hand, using the definition of the trace,  
\begin{equation*}
\int_\Omega \bar{\phi}(0^+,x) \bar{\rho}_0 = \lim_{t \to 0} \frac{1}{t} \iint_{[0,t] \times \Omega} \bar{\phi}(s,x)  \bar{\rho}_0(x)  \ddr s \ddr x.
\end{equation*}  
We fix some $t>0$. Due to the convergence of $\phi^{N}$ to $\bar{\phi}$, it clearly holds
\begin{equation*}
\lim_{N \to + \infty} \frac{1}{t} \iint_{[0,t] \times \Omega} \phi^{N}(s,x)  \bar{\rho}_0(x) \ddr s \ddr x = \frac{1}{t} \iint_{[0,t] \times \Omega} \bar{\phi}(s,x)  \bar{\rho}_0(x) \ddr s \ddr x. 
\end{equation*}  
For the value function $\phi^{N}$, we can use the information that we have on the temporal derivative, namely $\dr_t \phi^{N} \geqslant - \alpha^{N}$. It allows us to write, given the positivity of $\bar{\rho}_0$, 
\begin{align*}
\frac{1}{t} \iint_{[0,t] \times \Omega} \phi^{N}(s,x)  \bar{\rho}_0(x) \ddr s \ddr x  & = \frac{1}{t} \iint_{[0,t] \times \Omega} \left( \phi^{N}(0^+,x) + \int_0^s \dr_t \phi^{N}(r,x) \ddr r \right) \bar{\rho}_0(x) \ddr s \ddr x \\
& \geqslant \int_\Omega \phi^{N}(0^+, \cdot) \bar{\rho}_0 - \frac{1}{t} \iint_{[0,t] \times \Omega} s \alpha^{N}(s,x) \bar{\rho}_0(x) \ddr s \ddr x  \\
&\geqslant \int_\Omega \phi^{N}(0^+, \cdot)  \bar{\rho}_0(x) - \iint_{[0,t] \times \Omega}  \alpha^{N}(s,x) \bar{\rho}_0(x) \ddr s \ddr x.
\end{align*}
Now, recall that $\bar{\rho}_0 \leqslant 1$ and $\alpha^N$ converges as a measure to $\bar{p} + V$, hence 
\begin{equation*}
\frac{1}{t} \iint_{[0,t] \times \Omega} \bar{\phi}(s,x)  \bar{\rho}_0(x) \ddr s \ddr x = \lim_{N \to + \infty} \frac{1}{t} \iint_{[0,t] \times \Omega} \phi^{N}(s,x)  \bar{\rho}_0(x) \ddr s \ddr x \geqslant \limsup_{N \to + \infty} \left( \int_\Omega \phi^{N}(0^+, \cdot)  \bar{\rho}_0 \right) - (\bar{p} + V)([0,t] \times \Omega). 
\end{equation*}
Now we send $t$ to $0$, and use the fact that $(\bar{p} + V)(\{ 0 \} \times \Omega) = 0$ (this can be seen as a consequence of Corollary \ref{corollary_bound_discrete}) to conclude that 
\begin{equation*}
\limsup_{N \to + \infty} \int_\Omega \phi^{N}(0^+, \cdot)  \bar{\rho}_0 \leqslant \int_\Omega \bar{\phi}(0^+,x) \bar{\rho}_0,
\end{equation*}
which gives us the announced result. 
\end{proof}

To reach the conclusion of our main theorem, it is enough to show that $\bar{P}$ has the regularity we announced. But this easily derives from the weak convergence of $P^{N}$ to $\bar{P}$ and the estimates of Corollary \ref{corollary_bound_discrete} and Corollary \ref{corollary_bound_discrete_boundary}.

\begin{proof}[Proof of Theorem \ref{theorem_main}]
\label{page_proof_theorem}
For any smooth test functions $a,b$ (with $a$ being real-valued and $b$ being vector valued), given the convergence of $P^{N}$ to $\bar{p}$ it holds
\begin{equation*}
\begin{cases}
\dst{\iint_{[0,1] \times \Omega} a \bar{P}} & = \dst{\lim_{N \to + \infty} \left( \sum_{k=1}^{N-1} \tau \int_\Omega a(k \tau, \cdot) p^{N}_{k \tau} + \int_\Omega a(1, \cdot) p^N_1  \right)}, \\
\dst{\iint_{[0,1] \times \Omega} (\nabla \cdot b) \bar{P}} &= \dst{ - \lim_{N \to + \infty} \left( \sum_{k=1}^{N-1} \tau \int_\Omega g(k \tau, \cdot) \cdot \nabla(p^{N}_{k \tau}) + \int_\Omega b(1, \cdot) \cdot \nabla p^N_1  \right)},
\end{cases}
\end{equation*}
where on the second line we have done an integration by parts in the r.h.s. Using the estimates given by Corollary \ref{corollary_bound_discrete} and Corollary \ref{corollary_bound_discrete_boundary}, it is clear that
\begin{equation*}
\iint_{[0,1] \times \Omega} (\nabla \cdot b) \bar{P} \leqslant \int_0^1 \left( \| b(t, \cdot) \|_{L^2(\Omega)} \|\nabla V \|_{L^2(\Omega)} \right) \ddr t + \| b(1, \cdot) \|_{L^2(\Omega)} \|\nabla \Psi \|_{L^2(\Omega)}.
\end{equation*}
On the other hand, if $V, \Psi \in W^{1,q}$ with $q > d$ then, using the same propositions, 
\begin{equation*}
\iint_{[0,1] \times \Omega} a \bar{P} \leqslant C \int_0^1 \left( \| a(t, \cdot) \|_{L^1(\Omega)} \right) \ddr t + \| a (1, \cdot) \|_{L^\infty(\Omega)},
\end{equation*}
where $C$ depends only on $\| \nabla V \|_{L^q(\Omega)}, \| \nabla \Psi \|_{L^q(\Omega)}$ and $\Omega$. Standard functional analysis manipulations provide the conclusion of Theorem \ref{theorem_main}.
\end{proof}

\section*{Acknowledgments}

Both authors acknowledge the support of the French ANR via the contract MFG (ANR- 16-CE40-0015-01) and benefited from the support of the FMJH ``Program Gaspard Monge for Optimization and operations research and their interactions with data science'' and EDF via the PGMO project VarPDEMFG.

%
%
%
%

\end{document}